\documentclass[11pt,a4paper,reqno]{amsart}
\usepackage{amsmath,amssymb,amsthm,amscd}
\usepackage[all]{xy}
\usepackage{color}
\usepackage{url}

\author{Steffen Oppermann}
\address{Institutt for matematiske fag \\ NTNU \\ 7491 Trondheim \\ Norway}
\email{Steffen.Oppermann@math.ntnu.no}

\author{Jan \v S\v tov\'\i\v cek}
\address{Charles University in Prague \\ Faculty of Mathematics and Physics \\
Department of Algebra \\ Sokolovska 83 \\ 186 75 Praha 8 \\ Czech Republic}
\email{stovicek@karlin.mff.cuni.cz}

\title[Bounded derived category and perfect ghosts]%
{Generating the bounded derived category and perfect ghosts}
\date{\today}

\subjclass[2010]{18E30 (primary), 16E35, 14F05 (secondary)}
\keywords{Strongly finitely generated triangulated categories, thick subcategories, Ghost Lemma}

\thanks{The second named author is supported by the research project MSM~0021620839 and by the grant GA\v{C}R P201/10/P084.}

\renewcommand{\iff}{if and only if }
\newcommand{\st}{such that }
\newcommand{\Der}[1]{{\mathbf{D}({#1})}}
\newcommand{\Db}[1]{{\mathbf{D}^{\rm b}({#1})}}
\newcommand{\Dm}[1]{{\mathbf{D}^-({#1})}}
\newcommand{\Dle}[2]{{\mathbf{D}^{\le{#1}}({#2})}}
\newcommand{\Dge}[2]{{\mathbf{D}^{\ge{#1}}({#2})}}
\newcommand{\add}{\mathrm{add}}
\newcommand{\Prod}{\mathrm{Prod}}
\newcommand{\Cpx}[1]{{\mathbf{C}({#1})}}
\newcommand{\Cb}[1]{{\mathbf{C}^{\rm b}({#1})}}
\newcommand{\Cm}[1]{{\mathbf{C}^-({#1})}}
\newcommand{\Cle}[2]{{\mathbf{C}^{\le{#1}}({#2})}}

\newcommand{\modR}{\mathrm{mod}\hbox{-}R}

\newcommand{\modL}{\mathrm{mod}\hbox{-}\Lambda}
\newcommand{\projL}{\mathrm{proj}\hbox{-}\Lambda}
\newcommand{\perfL}{\mathrm{perf}\hbox{-}\Lambda}
\newcommand{\QcohX}{{\mathrm{Qcoh}\,\mathbb{X}}}
\newcommand{\cohX}{{\mathrm{coh}\,\mathbb{X}}}
\newcommand{\QcohP}{\mathrm{Qcoh}(\mathbf{P}^1_\mathbb{C})}
\newcommand{\vectX}{\mathrm{vect}\,\mathbb{X}}
\newcommand{\Ab}{\mathrm{Ab}}
\DeclareMathOperator{\fp}{fp}
\newcommand{\Z}{\mathbb{Z}}

\DeclareMathOperator{\Proj}{Proj}
\DeclareMathOperator{\Hom}{Hom}

\DeclareMathOperator{\Ext}{Ext}

\DeclareMathOperator{\Ker}{Ker}

\newcommand{\thick}[2]{\langle{#2}\rangle_{#1}}

\DeclareMathOperator{\gldim}{gldim}
\newcommand{\A}{\mathcal{A}}
\newcommand{\B}{\mathcal{B}}
\newcommand{\C}{\mathcal{C}}

\newcommand{\E}{\mathcal{E}}

\newcommand{\G}{\mathcal{G}}
\newcommand{\T}{\mathcal{T}}
\newcommand{\X}{\mathbb{X}}
\newcommand{\OX}{\mathcal{O}_\X}
\newdir{ >}{{}*!/-7pt/@{>}}     % Exercise 14, p. 8, K. H. Rose's XY-pic User's Guide 

\theoremstyle{plain}
\newtheorem{thm}{Theorem}
\newtheorem{prop}[thm]{Proposition}
\newtheorem{lem}[thm]{Lemma}
\newtheorem{cor}[thm]{Corollary}
\newtheorem{obs}[thm]{Observation}
\theoremstyle{definition}
\newtheorem{defn}[thm]{Definition}
\newtheorem{set}[thm]{Setup}
\theoremstyle{remark}
\newtheorem{rem}[thm]{Remark}

\begin{document}
\begin{abstract}
We show, for a wide class of abelian categories relevant in representation theory and algebraic geometry, that the bounded derived categories have no non-trivial strongly finitely generated thick subcategories containing all perfect complexes. In order to do so we prove a strong converse of the Ghost Lemma for bounded derived categories.
\end{abstract}
\maketitle

\section{Introduction}
\label{sec.intro}

Bondal and van den Bergh \cite{BvdB} have introduced the notion of a triangulated category being strongly finitely generated. They have shown that this property is useful when studying the representability of certain cohomological functors.

\begin{defn}[\cite{BvdB}]
Let $\mathcal{T}$ be a triangulated category. For $T \in \mathcal{T}$ let
\begin{align*}
& \thick{}{T} = \thick{1}{T} = \add \{T[i] \mid i \in \mathbb{Z} \} \text{, and} \\
& \thick{n+1}{T} = \add \{ {\rm cone}(f) \mid f \in \Hom_{\mathcal{T}}(\thick{}{T}, \thick{n}{T}) \}.
\end{align*}
The category $\mathcal{T}$ is called \emph{strongly finitely generated} if there is $T \in \mathcal{T}$ and $n \in \mathbb{N}$ such that $\mathcal{T} = \thick{n}{T}$.
\end{defn}

In Rouquier's terminology \cite{Rou} strongly finitely generated triangulated categories are those which have finite dimension.

The main result of this paper is that subcategories of bounded derived categories, which satisfy a natural assumption making sure that they are not too small (in particular they should not be the bounded derived category of an abelian subcategory), can only be strongly finitely generated if they actually coincide with the entire bounded derived category.

This result came as a surprise. One might expect subcategories of strongly finitely generated categories to be strongly finitely generated again. However here we show that actually the opposite is the more typical behavior.

More precisely we will show the following:

\begin{thm}[see Theorem~\ref{thm.strongfg_general}] \label{thm.generating_intro} ~

\begin{enumerate}
\item
Let $\Lambda$ be a noetherian algebra, that is a module finite algebra over a commutative noetherian ring. Let $\T$ be a thick subcategory of $\Db\modL$ such that $\perfL \subseteq \T$. If $\T$ is strongly finitely generated, then $\T = \Db\modL$.
\item
Let $\X$ be a projective scheme over a commutative noetherian ring $R$. Let $\mathcal{T}$ be a thick subcategory of $\Db\cohX$ such that $\Db\vectX \subseteq \T$. If $\T$ is strongly finitely generated then $\T = \Db\cohX$.
\end{enumerate}
\end{thm}

In many cases it is known that the bounded derived category is strongly finitely generated. If $\Lambda$ is an artin algebra, $\Db\modL$ is strongly finitely generated by~\cite[Proposition~7.37]{Rou}. If $\X$ is a separated scheme of finite type over a perfect field, $\Db\cohX$ is strongly finitely generated by~\cite[Theorem~7.38]{Rou} or~\cite[Theorem 6.3]{Lunts}. Thus for those cases we obtain the following consequence of Theorem~\ref{thm.generating_intro}:

\begin{cor} \label{cor.generating_intro}
\begin{enumerate}
\item
Let $\Lambda$ be an artin algebra. A thick subcategory $\mathcal{T}$ of $\Db\modL$ such that $\perfL \subseteq \T$ is strongly finitely generated if and only if $\T = \Db\modL$.
\item
Let $\X$ be a projective scheme of finite type over a finitely generated algebra over a perfect field. A thick subcategory $\mathcal{T}$ of $\Db\cohX$ such that $\Db\vectX \subseteq \T$ is strongly finitely generated if and only if $\T = \Db\cohX$.
\end{enumerate}
\end{cor}

Note that the condition in (2) just means that the scheme is both projective over a noetherian ring, and of finite type over a perfect field, and hence it is also separated over this field. Thus both our Theorem~\ref{thm.generating_intro}(2) and the above-mentioned result of \cite{Lunts,Rou} apply.

Theorem~\ref{thm.generating_intro}(1) and the corollary get a more concrete flavour if one compares them to several classification results on the thick subcategories $\T$ \st $\perfL \subseteq \T \subseteq \Db\modL$. Such subcategories correspond bijectively to thick subcategories of the Verdier quotient $\Db\modL/\perfL$, which is for Gorenstein algebras $\Lambda$ triangle equivalent to $\underline{\textrm{MCM}}(\Lambda)$, the stable category of maximal Cohen-Macaulay $\Lambda$-modules; see~\cite[Theorem 4.4.1]{Buch}. Let us mention a few classification results on subcategories of $\underline{\textrm{MCM}}(\Lambda)$, without any ambition to give a complete list. If $\Lambda = kG$ is a finite dimensional group algebra over a field $k$, we refer to~\cite[Theorem 3.4]{BCR}. For a commutative abstract hypersurface local ring $\Lambda$, the Main Theorem of~\cite{Taka} applies. Finally, an explicitly computed example for $\Lambda = k\langle x,y\rangle/(x^2,y^2,xy+yx)$ can be found in~\cite[Proposition 21]{Sto}. In all these setups, the references give an explicit list of thick subcategories of $\underline{\textrm{MCM}}(\Lambda)$, and thus of thick subcategories of $\Db\modL$ containing perfect complexes. Our Theorem~\ref{thm.generating_intro}(1) now shows that none of the non-trivial categories in these lists is strongly finitely generated.

The main ingredient for the proof of Theorem~\ref{thm.generating_intro} is a converse of the Ghost Lemma (see \cite{Beli} for some background on the Ghost Lemma): A map is called a \emph{ghost} for an object $M$ if the induced map between covariant $\Hom$-functors vanishes on $\{M[i] \mid i \in \mathbb{Z} \}$ (see Definition~\ref{def.ghost}). Then we show the following:

\begin{thm}[see Theorem~\ref{thm.anti-ghost}]
Let $M$ and $X$ be objects in $\Db{\modL}$ (resp.\ $\Db\cohX$) for $\Lambda$ (resp.\ $\X$) as in Theorem~\ref{thm.generating_intro}. Then $X \in \thick{n}{M}$ if and only if the composition of any sequence of $n$ maps
\[ \xymatrix{ X_n \ar[r] & X_{n-1} \ar[r] & \cdots \ar[r] & X_1 \ar[r] & X, } \]
all of which are $M$-ghosts, vanishes.
\end{thm}

The ``only if'' part of this theorem is the usual Ghost Lemma, the ``if'' part is the principal new result. Actually we prove (and need) a stronger statement (see Theorem~\ref{thm.anti-ghost}), which claims that it even suffices to look at ghosts in the smaller category $\perfL$ (resp.\ $\Db\vectX$).

This paper is organized as follows:

In Section~\ref{sec.setup} we fix and explain the general setup of our paper. We then prove that in both parts of Theorem~\ref{thm.generating_intro} we are in the situation of our general setup.

In Section~\ref{sec.products} we study what kind of products exist in right bounded derived categories. Furthermore we determine which objects are cocompact (that means their contravariant $\Hom$-functors send products to coproducts).

This is a central ingredient for the proof of our converse of the Ghost Lemma in Section~\ref{sec.main_result}. In that section we also show that this converse of the Ghost Lemma gives rise to a proof that certain subcategories of bounded derived categories can only be strongly finitely generated if they coincide with the entire bounded derived category.

% -----------------------------------------------------------------------------
\section{General setup}
\label{sec.setup}

In this section we fix the general setup for the rest of the paper. We then show that it generalizes both setups of Theorem~\ref{thm.generating_intro}.

\begin{set} \label{set.general}
Fix a commutative noetherian ring $R$, and $R$-categories $\E \subseteq \A$ such that:
\begin{enumerate}
\item $\A$ is a skeletally small abelian $R$-category and each $\Ext^i$-group ($i \ge 0$) is a finitely generated $R$-module.
\item $\E \subseteq \A$ is a full subcategory closed under extensions and kernels of epimorphisms in $\A$.
\item $\E$ is generating in $\A$, that is for any $A \in \A$ there is an epimorphism $E \to A$ for some $E \in \E$.
\item there is $n \in \mathbb{N}$ such that $\Ext_\A^i(E,F) = 0$ for any $i > n$ and $E,F \in \E$.
\end{enumerate}
\end{set}

\begin{rem} \label{rem.same_Ext}
We will see in Lemma~\ref{lem.Db(E)} that for $E,F \in \E$ one has $\Ext_\E^i(E,F) = \Ext_\A^i(E,F)$, where $\Ext_\E^i(E,F)$ is the Yoneda Ext in the exact category $\E$. Thus (4) of Setup~\ref{set.general} is equivalent to all sufficiently high extensions vanishing in $\E$, that is to $\gldim \E < \infty$.
\end{rem}

We can now give our most general result on strongly finitely generated subcategories of bounded derived categories.

\begin{thm} \label{thm.strongfg_general}
Let $\mathcal{E}$ and $\mathcal{A}$ be as in Setup~\ref{set.general}. Let $\mathcal{T}$ be a thick subcategory of $\Db\A$ such that $\Db\E \subseteq \mathcal{T}$. If $\mathcal{T}$ is strongly finitely generated then $\mathcal{T} = \Db\A$.
\end{thm}

We give a proof of this theorem at the end of this paper.

The following observation and proposition show that Setup~\ref{set.general} generalizes both setups of Theorem~\ref{thm.generating_intro}. In particular Theorem~\ref{thm.strongfg_general} implies Theorem~\ref{thm.generating_intro}.

\begin{obs}
Let $\Lambda$ be a noetherian algebra over a commutative noetherian ring $R$. Then for $\A = \modL$ and $\E = \projL$ the assumptions of Setup~\ref{set.general} are satisfied. Note that in this case $\Db\E = \perfL$.
\end{obs}

\begin{prop} \label{prop.setup_coh}
Let $\X$ be a projective scheme over a commutative noetherian ring $R$. Then for $\A = \cohX$ and $\E = \vectX$, the subcategory of all locally free coherent sheaves (also known as vector bundles), the assumptions of Setup~\ref{set.general} are satisfied.
\end{prop}

Before proving the proposition, we shortly discuss a technical issue. In the literature, the groups $\Ext^i(\mathcal{X},\mathcal{Y})$, where $\mathcal{X}, \mathcal{Y} \in \cohX$, are considered in various categories. We can take the Yoneda $\Ext$ in the category of all sheaves of $\OX$-modules, in the category of quasi-coherent sheaves, or in the category of coherent sheaves. It turns out that under suitable assumptions, the base category does not matter. Here we outline for convenience of the reader and also for future reference how to pass from $\QcohX$ to $\cohX$. First a definition:

\begin{defn}
Given a Grothendieck category $\B$, we call an object $B \in \B$ \emph{finitely presentable} if the functor $\Hom_\B(B,-)\colon \B \to \Ab$ commutes with direct limits. The category $\B$ is called \emph{locally coherent} if the full subcategory $\fp(\B)$ of finitely presentable objects of $\B$ is an abelian subcategory, and each $C \in \B$ is a direct limit of objects from $\fp(\B)$.
\end{defn}

Then we have the following well-known lemma:

\begin{lem} \label{lem.loc_coh}
Given any skeletally small abelian category $\A$, there is a locally coherent Grothendieck category $\B$ and a fully faithful exact functor $H\colon \A \to \B$, which induces an equivalence $\A \to \fp(\B)$. Moreover, such $\B$ is unique up to equivalence and the canonical functor $\Dm\A \to \Der\B$, where $\Der\B$ stands for the unbounded derived category of $\B$, is fully faithful.
\end{lem}

\begin{proof}
For the existence of $H\colon \A \to \B$ and the uniqueness of $\B$, we refer to~\cite[\S\S1.4 and 2.4]{CB}. Abusing notation as usual, we will identify $\A$ with the essential image of $H$.
For the last part, note that given any complex $M \in \Cpx\B$ whose homologies $H^i(M)$ belong to $\A$ for all $i \in \Z$ and vanish for $i\gg0$, one can, using standard arguments, construct a quasi-isomorphism $X \to M$ with $X \in \Cm\A$. It easily follows that $\Dm\A \to \Der\B$ is fully faithful (a more detailed argument in a different context will be given Lemma~\ref{lem.Db(E)}).
\end{proof}

Getting back to our situation, note that if $\X$ is a projective scheme over a commutative noetherian ring and $\A = \cohX$, we can take $\B = \QcohX$ and for $H$ the obvious inclusion. One can infer this from results in~\cite[\S II.5]{Hart} and especially from~\cite[Proposition II.5.15]{Hart}. Now, a direct consequence of Lemma~\ref{lem.loc_coh} is that $\Ext^i_\A(\mathcal{X},\mathcal{Y}) \cong \Ext^i_\B(\mathcal{X},\mathcal{Y})$ for each $\mathcal{X},\mathcal{Y} \in \A$ and $i \ge 0$.

\begin{proof}[Proof of Proposition~\ref{prop.setup_coh}]
Note that by assumption $\X = \Proj\frac{R[X_1, \ldots, X_n]}{I}$ for some $n$ and some homogeneous ideal $I$. The second point of Setup~\ref{set.general} is immediate and the third follows from~\cite[Proposition II.5.15]{Hart} and the fact that sheaves associated to free graded $\frac{R[X_1, \ldots, X_n]}{I}$-modules are locally free.

For the first point, one has $H^i(\X,\mathcal{Y}) \in \modR$ for any $\mathcal{Y} \in \cohX$ by~\cite[Theorem~III.5.2]{Hart}, where $H^i(\X,\mathcal{Y})$ is the sheaf cohomology as in~\cite[\S III.2]{Hart}. By~\cite[Proposition III.6.3]{Hart}, we have $H^i(\X,\mathcal{Y}) \cong \Ext^i(\OX,\mathcal{Y})$ for each $i \ge 0$, where the $\Ext$ groups are taken in the category of sheaves of $\OX$-modules.
Inspecting~\cite[Exercise III.3.6]{Hart}, one concludes that for noetherian schemes the same isomorphisms hold when the $\Ext$-groups are taken in the category $\QcohX$ (see also~\cite[Proposition III.2.5]{Hart}). Next, Lemma~\ref{lem.loc_coh} tells us that also
\[
\Ext^i_{\cohX}(\OX,\mathcal{Y}) \cong H^i(\X,\mathcal{Y}) \in \modR \quad
\textrm{ for each $\mathcal{Y} \in \cohX$ and $i \ge 0$.}
\]
Using a version of~\cite[Proposition III.6.7]{Hart} for quasi-coherent sheaves, we obtain for each $i \ge 0$, $\mathcal{L} \in \vectX$ and $\mathcal{Y} \in \cohX$:
\[
\Ext_{\cohX}^i(\mathcal{L}, \mathcal{Y}) \cong
\Ext_{\cohX}^i(\OX, \mathcal{L}^{\vee} \otimes \mathcal{Y}) \in \modR.
\]

Taking into account that the vector bundles generate $\cohX$, we will show by induction on $i \ge 0$ that $\Ext_{\cohX}^i(\mathcal{X}, \mathcal{Y}) \in \modR$ for every $\mathcal{X},\mathcal{Y} \in \cohX$. Namely, fix any short exact sequence
\[
0 \to \mathcal{K} \to \mathcal{L} \to \mathcal{X} \to 0
\]
with $\mathcal{L}$ locally free and apply $\Hom_\cohX(-,\mathcal{Y})$. If $i = 0$, the exact sequence
\[ 0 \to \Hom_\cohX(\mathcal{X}, \mathcal{Y}) \to \Hom_\cohX(\mathcal{L}, \mathcal{Y}) \]
provides evidence that $\Hom_\cohX(\mathcal{X}, \mathcal{Y})$ is a finitely generated $R$-module. For $i > 0$, consider the exact sequence
\[
\Ext_\cohX^{i-1}(\mathcal{K}, \mathcal{Y}) \to
\Ext_\cohX^i(\mathcal{X}, \mathcal{Y}) \to
\Ext_{\cohX}^i(\mathcal{L}, \mathcal{Y}).
\]
Then $\Ext_\cohX^{i-1}(\mathcal{K}, \mathcal{Y}) \in \modR$ by the inductive hypothesis and we know that $\Ext_{\cohX}^i(\mathcal{L}, \mathcal{Y}) \in \modR$ by the argument above. Hence $\Ext_\cohX^i(\mathcal{X}, \mathcal{Y})$ is a finitely generated $R$-module as well.

Finally by equality of \v{C}ech cohomology and sheaf cohomology, \cite[Theorem III.4.5]{Hart}, one easily sees that $H^i \equiv 0$ for $i > n$. Thus by the isomorphisms above, the final point of Setup~\ref{set.general} is satisfied.
\end{proof}

% -----------------------------------------------------------------------------
\section{Products and cocompact objects in $\Dm\A$}
\label{sec.products}

In order to prove the main results in Section~\ref{sec.main_result}, we first need to show that $\Dm\A$ has certain products of infinite families of objects, and to characterize cocompact objects in $\Dm\A$. We explain the latter concept in Definition~\ref{def.cocomp}.

As long as we are concerned with the abelian category of complexes $\Cm\A$, the situation with products and coproducts is easy. If a family of complexes $(X_i \mid i \in I)$ has a product or a coproduct, it is computed componentwise. This easily follows for instance from the adjoint formulas in~\cite[Lemma~3.1]{Gil}. It is an immediate observation that products and coproducts of the following (in general countably infinite) families always exist and coincide:

\begin{defn} \label{def.tends-down}
We say that a family $(X_i \mid i \in I)$ of complexes from $\Cm\A$ is \emph{descending} provided the following conditions are satisfied:
\begin{enumerate}
 \item There is $N \in \Z$ \st $X_i \in \Cle{N}{\A}$ for each $i \in I$.
 \item For each $n \in \Z$, there are only finitely many indices $i \in I$ \st the $n$-th component of $X_i$ is non-zero.
\end{enumerate}
\end{defn}

In fact, since we assume that $\A$ is a $\Hom$-finite category over a noetherian ring, it is not difficult to see that any family $(X_i \mid i \in I)$ which has a product or coproduct in $\Cm\A$ must be descending.

A little more tricky point is to see that we can compute products and coproducts of descending families in this way also in $\Dm\A$.
For coproducts, the situation is still covered by classical results. Namely, by Lemma~\ref{lem.loc_coh}, we can embed $\A$ into a locally coherent Grothendieck category $\B$ so that $\Dm\A \to \Der\B$ is fully faithful.
Since coproducts in $\B$ are exact, coproducts of quasi-isomorphisms are quasi-isomorphisms again. 
Using this fact one easily sees that coproducts in $\Der\B$ are computed componentwise.
It follows that the coproduct of a descending family of complexes from $\Dm\A$ belongs to $\Dm\A$ again.

Unfortunately, one cannot simply give a dual argument for products since products in $\B$ are in general not exact. For instance, the exactness is well-known to fail for $\B = \QcohP$; see~\cite[Example 4.9]{Krause}. However, using fully the assumptions in Setup~\ref{set.general}, we still can prove:

\begin{prop} \label{prop.prod}
If $(X_i \mid i \in I)$ is a descending family of complexes of $\Dm\A$, then the componentwise product of the family is a product in the category $\Dm\A$.

On the other hand, if $\prod_{i \in I} X_i$ exists in $\Dm\A$ for some family $(X_i \mid i \in I)$, then there is a family $(X'_i \mid i \in I)$ of isomorphic complexes in $\Dm\A$ \st $(X'_i \mid i \in I)$ is a descending family.
\end{prop}

We need a to prove a few lemmas first. We begin by describing the relations between the derived categories of $\E$ and $\A$. Here we consider $\E$ as an exact category with the exact structure induced from $\A$, so quasi-isomorphisms between complexes over $\E$ are those chain complex morphisms whose mapping cones are acyclic in $\E$.

\begin{lem} \label{lem.Db(E)}
In the following diagram with the canonical functors, all the functors are fully faithful and the upper one is a triangle equivalence:
\[
\begin{CD}
\Dm\E @>>> \Dm\A  \\
@AAA       @AAA   \\
\Db\E @>>> \Db\A
\end{CD}
\]
\end{lem}

\begin{proof}
Note that since $\E$ is generating in $\A$, for any complex $M$ in $\Cm{\A}$ there is a quasi-isomorphism $X \to M$ with $X \in \Cm{\E}$; see~\cite[Lemma~I.4.6]{Hart2}. Hence the upper horizontal functor is essentially surjective.

To see that the upper horizontal functor is fully faithful, let $X, Y \in \Cm{\E}$, and let $f \in \Hom_{\Dm{\A}}(X, Y)$. Then $f$ is represented by a right fraction $g\sigma^{-1}$, where $\sigma$ is a quasi-isomorphism. That is, there is $M \in \Cm{\A}$ such that $g\sigma^{-1}$ is in the upper row of the following diagram.
\[ \xymatrix{
X && \ar[ll]^{\rm qis}_{\sigma} M \ar[rr]^{g} && Y \\
&& Z \ar[u]_{\rm qis}^{\tau}
} \]
Now there exists a quasi-isomorphism $\tau$ from $Z \in \Cm{\E}$ to $M$ by the discussion above, and we have
\[ f = g \sigma^{-1} = (g \tau) (\sigma \tau)^{-1} \in \Hom_{\Dm{\E}}(X, Y). \]
Similarly one sees that if a map vanishes in $\Dm{\A}$, then it already vanishes in $\Dm{\E}$, whence the upper horizontal functor is faithful.

To see that the left vertical functor is fully faithful, note that for any quasi-isomorphism $Y \to M$ with $Y \in \Cb{\E}$ and $M \in \Cm{\E}$, there is a quasi-isomorphism $\tau\colon M \to W$ with $W \in \Cb{\E}$ (given by truncation in $\Cm\A$ -- this is possible since $\E$ is closed under kernels of epimorphisms in $\A$). Hence for each left fraction $f = \sigma^{-1}g \in \Hom_{\Dm{\E}}(X, Y)$ with $X,Y \in \Db{\E}$ we have $\sigma^{-1}g = (\tau \sigma)^{-1} (\tau g) \in \Hom_{\Db{\E}}(X, Y)$. One also easily checks that $\sigma^{-1}g$ vanishes in $\Hom_{\Dm{\E}}(X,Y)$ \iff it vanishes in $\Hom_{\Db{\E}}(X, Y)$.

The fact that the right vertical functor is fully faithful is classical (and can be seen similarly). It follows from the diagram that also the lower horizontal functor is fully faithful.
\end{proof}

Therefore, it suffices to prove the existence of products of descending families in $\Dm\E$. Here we aim to exploit the fact that by Setup~\ref{set.general}(4) and in view of Remark~\ref{rem.same_Ext}, $\E$ has finite global dimension as an exact category. Let us first establish a basic property of Yoneda $\Ext$-groups in $\E$.

\begin{lem} \label{lem.Yoneda-chains}
Assume we have $n \ge 1$ and two exact sequences
\[ \varepsilon_i\colon \quad 0 \to Y \to E_{1,i} \to E_{2,i} \to \dots \to E_{n,i} \to X \to 0 \qquad (i=1,2) \]
in $\E$ \st $[\varepsilon_1] = [\varepsilon_2]$ in $\Ext^n_\E(X,Y)$. Then there is a commutative diagram of the following form with rows exact in $\E$ and inflations in all columns:
\[
\xymatrix@=22pt{
\varepsilon_1\colon \quad
0 \ar[r] &
Y \ar[r] \ar@{=}[d] &
E_{1,1} \ar[r] \ar@{ >->}[d] &
E_{2,1} \ar[r] \ar@{ >->}[d] &
\dots \ar[r] &
E_{n,1} \ar[r] \ar@{ >->}[d] &
X \ar[r] \ar@{=}[d] &
0
\\
\,\,\eta\colon \quad
0 \ar[r] &
Y \ar[r] &
F_1 \ar[r] &
F_2 \ar[r] &
\dots \ar[r] &
F_n \ar[r] &
X \ar[r] &
0
\\
\varepsilon_2\colon \quad
0 \ar[r] &
Y \ar[r] \ar@{=}[u] &
E_{1,2} \ar[r] \ar@{ >->}[u] &
E_{2,2} \ar[r] \ar@{ >->}[u] &
\dots \ar[r] &
E_{n,2} \ar[r] \ar@{ >->}[u] &
X \ar[r] \ar@{=}[u] &
0
}
\]
\end{lem}

\begin{proof}
Since $[\varepsilon_1] = [\varepsilon_2]$, there is by definition a finite collection $\eta_0, \dots, \eta_m$ of exact sequences in $\E$ with $n$ middle terms together with morphisms
\[
\xymatrix@!=9pt{
& \eta_1 &&
\eta_3 &&&&&
\eta_{m-1}
\\
\varepsilon_1 = \eta_0 \;\quad \ar[ur] &&
\eta_2 \ar[ul] \ar[ur] &&
\eta_4 \ar[ul] \ar[ur] &&
\ar@{}[ul]|{\dots} &
\eta_{m-2} \ar[ul] \ar[ur] &&
\quad\;\; \eta_m = \varepsilon_2, \ar[ul]
}
\]
\st the components at $X$ and $Y$ are the identity morphisms.

First note that the chain of morphisms can be taken so that all components in the morphisms are inflations. Indeed, we can replace each morphism
\[
\xymatrix@=22pt{
\eta_{i\pm 1}\colon \ar[d]_f &
0 \ar[r] &
Y \ar[r] \ar@{=}[d] &
F_1 \ar[r] \ar[d] &
F_2 \ar[r] \ar[d] &
\dots \ar[r] &
F_n \ar[r] \ar[d] &
X \ar[r] \ar@{=}[d] &
0
\\
{_{\phantom{\pm 1}}}\eta_i\colon &
0 \ar[r] &
Y \ar[r] &
F'_1 \ar[r] &
F'_2 \ar[r] &
\dots \ar[r] &
F'_n \ar[r] &
X \ar[r] &
0
}
\]
by the morphism $f' = (f, q_2, \dots, q_n)^t\colon \eta_{i\pm 1} \to \eta_i \oplus \xi_2 \oplus \dots \oplus \xi_n$, where $q_\ell$ stands for the morphism
\[
\xymatrix@=22pt{
\eta_{i\pm 1}\colon \ar[d]_{q_\ell} &
0 \ar[r] &
Y \ar[r] \ar[d] &
\dots \ar[r] &
F_{\ell-1} \ar[r] \ar[d] &
F_\ell \ar[r] \ar@{=}[d] &
\dots \ar[r] &
X \ar[r] \ar[d] &
0
\\
{_{\phantom{\pm 1}}}\xi_\ell\colon &
0 \ar[r] &
0 \ar[r] &
\dots \ar[r] &
F_\ell \ar[r]^{1_{F_\ell}} &
F_\ell \ar[r] &
\dots \ar[r] &
0 \ar[r] &
0
}
\]
Of course, if for a fixed $i$ we do such a substitution for $f\colon \eta_{i-1} \to \eta_i$, we must also replace the original adjacent morphism $g\colon \eta_{i+1} \to \eta_i$ by $g' = (g, 0, \dots, 0)^t\colon \eta_{i+1} \to \eta_i \oplus \xi_2 \oplus \dots \xi_n$. Similarly, if we replace the morphism $\eta_{i+1} \to \eta_i$, we must correspondingly change $\eta_{i-1} \to \eta_i$. However, this does not pose any problem since if (in the notation above) $g$ is an inflation, so will be $g'$.

Now, if $m=2$, we are done. If $m>2$, we construct the pushout diagram
\[
\xymatrix@!=9pt{
& \tilde\eta & \\
\eta_1 \ar[ur] && \eta_3 \ar[ul] \\
& \eta_2 \ar[ul] \ar[ur] &
}
\]
and replace the original chain of morphisms by
\[
\xymatrix@!=9pt{
& \tilde\eta &&&&&
\eta_{m-1}
\\
\varepsilon_1 = \eta_0 \;\quad \ar[ur] &&
\eta_4 \ar[ul] \ar[ur] &&
\ar@{}[ul]|{\dots} &
\eta_{m-2} \ar[ul] \ar[ur] &&
\quad\;\; \eta_m = \varepsilon_2. \ar[ul]
}
\]
After finitely many repetitions, we reduce the length of the chain to two.
\end{proof}

Next we prove a crucial lemma about left fractions in $\Dm\E$.

\begin{lem} \label{lem.cut-middle}
Let $d = \gldim\E$, $N \in \Z$, and $f\colon Y \to X$ be a morphism in $\Dm\E$ \st $X \in \Dle{N}\E$. Then $f$ can be represented by a fraction
\[
\xymatrix{
& Z
\\
Y \ar[ur]^g &&
X \ar[ul]_\sigma^{\rm qis}
}
\]
with $Z \in \Dle{N+d}\E$.
\end{lem}

\begin{proof}
Let us take any left fraction $\tau^{-1} g'$ representing $f$, where $\tau\colon X \to W$ is a quasi-isomorphism, and suppose that $W \in \Dle{N+d'}\E$ for some integer $d' > d$. Since $H^i(X) = 0$ for all $i > N$ when considering $X$ as a complex in $\Dm\A$, the sequence
\[
\varepsilon_1: \quad 0 \to Z^N(W) \to W^N \to \dots \to W^{N+d'-1} \to W^{N+d'} \to 0
\]
is exact in $\A$. Using the assumption of Setup~\ref{set.general}(2) that $\E$ is closed under kernels of epimorphisms in $\A$, we get $Z^i(W) \in \E$ for all $i \ge N$. We can view $\varepsilon_1$ as a representative of an element of $\Ext^{d'}_\E\big(W^{N+d'}, Z^N(W)\big)$. Now the equivalence class of $\varepsilon_1$ must vanish in $\Ext^{d'}_\E\big(W^{N+d'}, Z^N(W)\big)$ since $d' > d = \gldim\E$. Using Lemma~\ref{lem.Yoneda-chains} for $\varepsilon_1$ and any $\varepsilon_2$ in which the last morphism onto $W^{N+d'}$ splits, we obtain a commutative diagram in $\E$ of the form
\[
\xymatrix@=22pt{
0 \ar[r] &
Z^N(W) \ar[r] \ar@{=}[d] &
W^N \ar[r] \ar@{ >->}[d]_{j_1} &
\dots \ar[r] &
W^{N+d'-1} \ar[r] \ar@{ >->}[d]_{j_{d'}} &
W^{N+d'} \ar[r] \ar@{=}[d] &
0\phantom{,}
\\
0 \ar[r] &
Z^N(W) \ar[r] &
F_1 \ar[r] &
\dots \ar[r] &
F_{d'} \ar[r]^q &
W^{N+d'} \ar[r] &
0,
}
\]
where the rows are exact and $q$ splits. Thus the inflation $k\colon \Ker q \to F_{d'}$ splits and there is a morphism $r\colon F_{d'} \to \Ker q$ \st $rk = 1_{\Ker q}$. The chain complex morphism $\upsilon\colon W \to W'$ defined by the diagram
\[
\xymatrix@R=22pt@C=14pt{
\dots \ar[r] &
W^{N-1} \ar[r] \ar@{=}[d] &
W^N \ar[r] \ar@{ >->}[d]_{j_1} &
\dots \ar[r] &
W^{N+d'-2} \ar[r] \ar@{ >->}[d]_{j_{d'-1}} &
W^{N+d'-1} \ar[r] \ar@{ >->}[d]_{rj_{d'}} &
W^{N+d'} \ar[d] \ar[r] &
\dots
\\
\dots \ar[r] &
W^{N-1} \ar[r] &
F_1 \ar[r] &
\dots \ar[r] &
F_{d'-1} \ar@{->>}[r] &
\Ker q \ar[r] &
0 \ar[r] &
\dots
}
\]
is a quasi-isomorphism in $\Dm\E$ since it is a quasi-isomorphism in $\Dm\A$, and since the natural functor from $\Dm\E$ to $\Dm\A$ is an equivalence by Lemma~\ref{lem.Db(E)}. It follows $\tau^{-1} g' = (\upsilon\tau)^{-1} \upsilon g'$ in $\Hom_\Dm{\E}(Y,X)$ and the middle term $W'$ of the latter fraction belongs to $\Dle{N+d'-1}\E$. After repeating the procedure finitely many times, we obtain an equivalent fraction with middle term in $\Dle{N+d}\E$.
\end{proof}

Now we are in a position to prove the existence of products.

\begin{proof}[Proof of Proposition~\ref{prop.prod}]
Assume first we have a descending family $(X_i \mid i \in I)$ of objects of $\Dm\E$. Let $f_i\colon Y \to X_i$ be any collection of morphisms. By Lemma~\ref{lem.cut-middle} we can represent these by fractions
\[
\xymatrix{
& Z_i
\\
Y \ar[ur]^{g_i} &&
X_i \ar[ul]_{\sigma_i}^{\rm qis}
}
\]
\st $(Z_i \mid i \in I)$ is a descending family. Hence we can construct the product morphism $f\colon Y \to \prod_{i \in I} X_i$ as the fraction
\[
\xymatrix{
& \prod Z_i
\\
Y \ar[ur]^{(g_i)} &&
\prod X_i \ar[ul]_{\prod \sigma_i}^{\rm qis}.
}
\]
Here, $\prod$ denotes the componentwise products. It is straightforward to show that this product morphism is unique. Thus we have shown that the componentwise product of a descending family in $\Dm\E$ is its product.

If $(X_i \mid i \in I)$ is a descending family in $\Dm\A$, \cite[Lemma~I.4.6]{Hart2} yields for each $i \in I$ a quasi-isomorphism $\tau_i\colon X'_i \to X_i$ with $X'_i \in \Dm\E$. Moreover, we can take the morphisms so that $(X'_i \mid i \in I)$ is a descending family. Clearly, the componentwise product $\prod \tau_i$ is a quasi-isomorphism, so the componentwise product $\prod X_i$ is really a product in $\Dm\A$ by Lemma~\ref{lem.Db(E)}.

Conversely, let us assume that $(X_i \mid i \in I)$ is a family of objects of $\Dm\A$ \st the product $\prod_{i \in I} X_i$ exists in $\Dm\A$. First we show that there is some $N \in \Z$ \st $H^n(X_i) = 0$ for each $n \ge N$ and $i \in I$. Indeed, if for arbitrarily large numbers $n$ there were $i \in I$ \st $H^n(X_i) \ne 0$, it would mean (given the fact that $H^n\colon \Dm\A \to \A$ are additive functors) that $H^n(\prod X_i) \ne 0$ for arbitrary large $n$. This is absurd.

It remains to check that (up to isomorphism) the family $(X_i \mid i \in I)$ satisfies the second point of Definition~\ref{def.tends-down}.
Let us by way of contradiction assume that there is $n \in \Z$ \st $H^n(X_{i_m}) \ne 0$ for infinitely many indices $i_1$, $i_2$, $i_3$, \dots. Then the chain of split epimorphisms
\[
\prod_{i \in I} X_i \to \prod_{i \in I \setminus \{i_1\}} X_i \to \prod_{i \in I \setminus \{i_1,i_2\}} X_i \to \dots
\]
yields a chain of proper split epimorphism
\[
H^n(\prod_{i \in I} X_i) \to H^n(\prod_{i \in I \setminus \{i_1\}} X_i) \to H^n(\prod_{i \in I \setminus \{i_1,i_2\}} X_i) \to \dots
\]
in $\A$. Applying $\Hom_\A\big(H^n(\prod_{i \in I} X_i), -\big)$ gives us an infinite chain of proper split epimorphisms between finitely generated $R$-modules, which is impossible. Therefore, for each $n \in \Z$ there are only finitely many indices $i \in I$ \st $H^n(X_i) \ne 0$ and it is easy to see that $(X_i \mid i \in I)$ can be replaced by an isomorphic family $(X'_i \mid i \in I)$ which is descending.
\end{proof}

If we have an additive category $\T$ with some infinite coproducts, we can consider compact objects. These are defined as those $Y \in \T$ for which $\Hom_\T(Y,-)$ commutes with all coproducts which exist in $\T$, and they play a very important role in the theory of triangulated categories (see for instance~\cite{Nee}). In this paper, the dual concept is relevant:

\begin{defn} \label{def.cocomp}
Let $\T$ be an additive category. An object $Y \in \T$ is \emph{cocompact} if $Y$ is compact in the opposite category $\T^\mathrm{op}$. That is, given any family $(X_i \mid i \in I)$ of objects of $\T$ \st the product $\prod_{i\in I} X_i$ exists in $\T$ (and so the coproduct $\coprod_{i\in I} X_i$ exists in $\T^\mathrm{op}$), then the canonical group homomorphism
\[ \coprod_{i \in I} \Hom_\T(X_i, Y) \to \Hom_\T\big(\prod_{i\in I} X_i, Y\big) \]
is an isomorphism. Equivalently, we can say that $Y$ is cocompact if any morphism $f\colon \prod_{i \in I} X_i \to Y$ in $\T$ factors through the canonical projection $\prod_{i \in I} X_i \to \prod_{i \in J} X_i$ for some finite subset $J \subseteq I$.
\end{defn}

The main result of this section, which is crucial for the next section, is the following characterization of cocompact objects in $\Dm\A$.

\begin{thm} \label{thm.cocomp}
Assume Setup~\ref{set.general}. Then an object $Y \in \Dm\A$ is cocompact \iff $Y$ is isomorphic to a bounded complex over $\A$.
\end{thm}

\begin{proof}
For the if-part, assume that $Y$ is a bounded complex and $N$ is an integer \st $Y \in \Dge{N}\A$. If $(X_i \mid i \in I)$ is a family of complexes in $\Dm\A$ which has a product, it is up to isomorphism a descending family by Proposition~\ref{prop.prod}. Then, however, all but finitely many $X_i$ belong to $\Dle{N-1}\A$ and so does their product. Using the well-known fact that
\[ \Hom_{\Dm\A}\big( \Dle{N-1}\A, \Dge{N}\A \big) = 0, \]
every morphism $f\colon \prod_{i \in I} X_i \to Y$ factors through the finite product $\prod_{i \in J} X_i$ formed by those $X_i$ which are not isomorphic to a complex in $\Dle{N-1}\A$.

Conversely, assume that $Y$ is cocompact in $\Dm\A$ and fix $N'$ such that $Y \in \Dle{N'}\A$. Consider for each $i \le N'$ the obvious morphism $f_i\colon Z^i(Y)[-i] \to Y$. By Proposition~\ref{prop.prod}, the product $\prod Z^i(Y)[-i]$ exists in $\Dm\A$ and is computed componentwise. Since we assume that $Y$ is cocompact, all but finitely many $f_i$ must vanish in $\Dm\A$, which easily implies that $Y$ has only finitely many non-zero homologies. In particular, $Y$ is isomorphic to a bounded complex.
\end{proof}

% -----------------------------------------------------------------------------
\section{Main results}
\label{sec.main_result}

\begin{defn} \label{def.ghost}
Given $M \in \Dm\A$, we say that a morphism $f\colon X \to Y$ is a \emph{(covariant) $M$-ghost} provided
\[ \Hom_{\Dm\A}(f, M[i]) = 0 \quad \textrm{for each $i \in \Z$.} \]
The class of all covariant $M$-ghosts in $\Dm\A$ will be denoted by $\G^M$.
\end{defn}

\begin{rem}
Dually one can define contravariant ghosts. In many papers only contravariant ghosts are considered (and then they are just called ghosts). However in our setup it is more convenient to work with covariant ghosts.
\end{rem}

The notion of ghosts is closely related to a more general representation-theoretic notion of approximations.

\begin{defn} \label{def.approx}
Let $\T$ be a category and $\C \subseteq \T$ be a full subcategory. A morphism $g\colon Y \to C$ in $\T$ is called a \emph{left $\C$-approximation} of $Y$ if $C \in \C$ and any morphism $g'\colon Y \to C'$ with $C' \in \C$ factors through $f$. The subcategory $\C$ is said to be \emph{covariantly finite} in $\T$ provided that every $Y \in \T$ admits some left $\C$-approximation $g\colon Y \to C$.
\end{defn}

To be more specific about the relation between ghosts and approximations, observe that given $Z \in \thick{}{M}$ and a triangle
\[ \xymatrix{ X \ar[r]^f & Y \ar[r]^g & Z \ar[r] & X[1] }, \]
then $f$ is a covariant $M$-ghost \iff $g$ is a left $\thick{}{M}$-approximation.

The Ghost Lemma relates $M$-ghosts and objects in $\thick{n}{M}$:

\begin{lem}[{Ghost Lemma -- see for instance \cite[Lemma~2.2]{Beli}}] \label{lem.classic_ghost}
Let $X$ and $M \in \Dm\A$.
\begin{enumerate}
\item If $X \in \thick{n}{M}$ then $(\G^M)^n(-, X) = 0$.
\item If $\thick{}{M}$ is covariantly finite in $\Dm\A$ then the converse holds.
\end{enumerate}
\end{lem}

\begin{rem}
\begin{enumerate}
\item The Ghost Lemma holds for any triangulated category, not just $\Dm\A$.
\item First versions of the Ghost Lemma have appeared in \cite{Kelly,Street,Chris}. More recently published versions include~\cite[Corollary 5.5]{Beli2} and~\cite[Lemma~4.11]{Rou}.
\end{enumerate}
\end{rem}

One main result of this paper, showing the converse of the Ghost Lemma in our setup without having covariant finiteness, is as follows:

\begin{thm} \label{thm.anti-ghost}
Let $\E \subseteq \A$ be as in Setup~\ref{set.general}, and fix $X,M \in \Db\A$ and $n \ge 0$. Then the following are equivalent:
\begin{enumerate}
  \item $X \in \thick{n}M$,
  \item $(\G^M)^n(-,X) = 0$,
  \item $(\G^M \!|_{\Db\A})^n(-, X) = 0$,
  \item $ \Hom_{\Dm\A}(\Db\E, X) \circ (\G^M \!|_{\Db\E})^n  = 0$.
\end{enumerate}
\end{thm}

Before proving the theorem, we introduce some notation and make a simple observation. For a complex $Y \in \Dm\A$, we denote by $\tau^{\geq i} Y$ the complex given by
\[ (\tau^{\geq i} Y)^j = \left\{ \begin{array}{ll} Y^j & \text{ for } j \geq i \\ 0 & \text{ otherwise} \end{array} \right. , \]
and similarly for $\tau^{\leq i} Y$. The complexes $\tau^{\geq i} Y$ and $\tau^{\leq i} Y$ are called \emph{brutal truncations} of $Y$. Note that for any $Y \in \Dm\E$ and any $i \in \mathbb{Z}$ there is a triangle
\[ \xymatrix{ \tau^{\geq i} Y \ar[r] & Y \ar[r] & \tau^{\leq i-1} Y \ar[r] & \tau^{\geq i} Y[1] }. \]

\begin{lem} \label{lem.truncating-morph}
Let $X \in \Db\A$ and $Y \in \Dm\A$. Then for $i \ll 0$, the morphism $\tau^{\geq i} Y \to Y$ yields an isomorphism

\[ \Hom_{\Dm\A}(Y,X) \overset{\cong}\longrightarrow \Hom_{\Dm\A}(\tau^{\geq i} Y,X). \]
\end{lem}

\begin{proof}
Invoking again the fact that $\Hom_{\Dm\A}\big( \Dle{i-1}\A, \Dge{i}\A \big) = 0$, observe that
\[ \Hom_{\Dm\A}(\tau^{\leq i-1}Y, X) = 0 = \Hom_{\Dm\A}(\tau^{\leq i-1}Y, X[1]) \]
for $i \ll 0$. The statement immediately follows from the existence of the triangles $\tau^{\geq i} Y \to Y \to \tau^{\leq i-1} Y \to \tau^{\geq i}[1]$.
\end{proof}

Now we return to the proof of the theorem.

\begin{proof}[Proof of Theorem~\ref{thm.anti-ghost}]
(1) $\implies$ (2) is the usual Ghost Lemma (Lemma~\ref{lem.classic_ghost}(1)).

(2) $\implies$ (3) and (3) $\implies$ (4) are trivial.

(4) $\implies$ (2). Assume, conversely to (2), that we have a non-zero composition
\[
\xymatrix{
X_n \ar[r] &
X_{n-1} \ar[r] &
\dots \ar[r] &
X_1 \ar[r] &
X
}
\]
of $M$-ghosts in $\Dm\E$ (since $\Dm\E \cong \Dm\A$; here we replace $X$ by an isomorphic object which lies in $\Dm\E$). Possibly replacing the $X_i$ by isomorphic objects we may assume that the maps above are represented by chain complex morphisms (rather than fractions).

We reduce this chain inductively to a non-zero composition of $M$-ghosts in $\Db\E$.

First we use Lemma~\ref{lem.truncating-morph} to truncate $X_n$ so that $\Hom_{\Dm\A}(\tau^{\geq i_n}X_n, X) \cong \Hom_{\Dm\A}(X_n, X)$. In particular, the induced map from $\tau^{\geq i_n} X_n$ to $X$ is still non-zero.

Now note that for $i \leq i_n$ we obtain an induced map $\xymatrix{ \tau^{\geq i_n} X_n \ar[r] & \tau^{\geq i} X_{n-1}}$ making the following diagram commutative.
\[
\xymatrix{
X_n \ar[r] &
X_{n-1} \ar[r] &
\dots \ar[r] &
X_1 \ar[r] &
X_0
\\
\tau^{\geq i_n} X_n \ar[u] \ar[r] &
\tau^{\geq i} X_{n-1} \ar[u]
}
\]
We want this map $\xymatrix{ \tau^{\geq i_n} X_n \ar[r] & \tau^{\geq i} X_{n-1}}$ to also be an $M$-ghost. Set
\[ J = \{j \in \mathbb{Z} \mid \Hom_{\Db\A}(\tau^{\geq i_n} X_n, M[j]) \neq 0\}. \]
Note that $\Hom_{\Db\A}(\tau^{\geq i_n} X_n, M[j]) = 0$ for $j \ll 0$, and by Lemma~\ref{lem.cut-middle} also for $j \gg 0$, so the set $J$ is finite. Now, using Lemma~\ref{lem.truncating-morph} one sees that for $i \ll 0$ we have
\[ \Hom_{\Dm\A}(\tau^{\geq i} X_{n-1}, M[j]) \cong \Hom_{\Dm\A}(X_{n-1}, M[j]) \quad \forall j \in J. \]
Hence, since the map $\xymatrix{ \tau^{\geq i_n} X_n \ar[r] & X_{n-1}}$ is an $M$-ghost, there is $i_{n-1}$ such that the map $\xymatrix{ \tau^{\geq i_n} X_n \ar[r] & \tau^{\geq i_{n-1}} X_{n-1}}$ also is an $M$-ghost.

Going on inductively, we construct a commutative diagram
\[
\xymatrix{
X_n \ar[r] &
X_{n-1} \ar[r] &
\dots \ar[r] &
X_1 \ar[r] &
X
\\
\tau^{\geq i_n} X_n \ar[u] \ar[r] &
\tau^{\geq i_{n-1}} X_{n-1} \ar[u] \ar[r] &
\dots \ar[r] &
\tau^{\geq i_1} X_1 \ar[u] \ar[r] &
\tau^{\geq i_0} X \ar[u]
}
\]
\st the objects below belong to $\Db\E$, all horizontal morphisms are $M$-ghosts, and the composition from the left lower to the right upper corner of the diagram is non-zero. This is exactly the kind of sequence of $M$-ghosts claimed to not exist in (4).

(2) $\implies$ (1). We will show that in $\Dm\A$ we have $X \in \thick{n}{\Prod \thick{}{M}}$, where $\Prod \thick{}{M}$ denotes the category whose objects are products of objects in $\thick{}{M}$, and then use a compactness argument from~\cite{BvdB} to get $X \in \thick{n}{M}$.

To this end, we first claim that $\Prod \thick{}{M}$ is covariantly finite in $\Dm\A$. That is, we must construct a left $\Prod \thick{}{M}$-approximation for any fixed $Y \in \Dm\A$. Notice that given any $i \in \Z$, $\Hom_{\Dm\A}(Y, M[i])$ is a finitely generated $R$-module. Indeed, $\Hom_{\Dm\A}(Y, M[i]) \cong \Hom_{\Dm\A}(\tau^{\geq \ell} Y, M[i])$ for $\ell \ll 0$ by Lemma~\ref{lem.truncating-morph}, and the latter is a finitely generated $R$-module by Setup~\ref{set.general}(1). Having fixed any collection $g_{i,1}, \dots, g_{i,a_i}\colon X \to M[i]$ of generators of that $R$-module (for a suitable $a_i \geq 0$), it is a well-known and easily checked fact that the morphism
\[ g_i = (g_{i,1}, \dots, g_{i,a_i})^t\colon Y \longrightarrow M[i]^{a_i} \]
is a left $\add M[i]$-approximation of $Y$. Recall further that $\Hom_{\Dm\A}(Y, M[i]) = 0$ for $i \ll 0$, so that we can take $a_i = 0$ for $i \ll 0$. Therefore, the product $\prod_{i \in \Z} M[i]^{a_i}$ exists in $\Dm\A$ and it is computed componentwise (see Proposition~\ref{prop.prod}). A little more standard checking, using the universal property of products, shows that the morphism
\[ g\colon Y \longrightarrow \prod_{i \in \Z} M[i]^{a_i}, \]
whose components are the above defined maps $g_i\colon Y \to M[i]^{a_i}$, is a left $\Prod \thick{}{M}$-approx\-imation of $Y$. This finishes the proof of the claim.

Now it follows that $X \in \thick{n}{\Prod \thick{}{M}}$ by essentially the same arguments as used for Lemma~\ref{lem.classic_ghost}(2). We quickly recall them for the convenience of the reader.

First construct a sequence of triangles
\begin{align*}
X_{\phantom{0}} \longrightarrow \prod_{i \in \Z} M[i]^{a_i} &\longrightarrow X_1 \overset{f_1}{\longrightarrow} X[1]_{\phantom{0}} \\
X_1 \longrightarrow \prod_{i \in \Z} M[i]^{b_i} &\longrightarrow X_2 \overset{f_2}{\longrightarrow} X_1[1], \\
& \quad\!\vdots \\
X_{n-1} \longrightarrow \prod_{i \in \Z} M[i]^{z_i} &\longrightarrow X_n \overset{f_n}{\longrightarrow} X_{n-1}[1],
\end{align*}
where the leftmost morphisms are left $\Prod \thick{}{M}$-approximations. It follows that the rightmost morphisms, labeled $f_i$, are $M$-ghosts.

Now, using the octahedral axiom, one sees that the cone of $f_1[-1] \circ f_2[-2]$ is an extension of $\prod_{i \in \Z} M[i]^{a_i}$ and $\prod_{i \in \Z} M[i-1]^{b_i}$, and thus in $\thick{2}{\Prod \thick{}{M}}$. Iterating this one sees that the cone of $f_1[-1] \circ f_2[-2] \circ \cdots \circ f_n[-n]$ lies in $\thick{n}{\Prod \thick{}{M}}$. But $f_1[-1] \circ f_2[-2] \circ \cdots \circ f_n[-n] \in (\G^M)^n(X_n[-n], X) = 0$. So this triangle splits, and $X \in \thick{n}{\Prod \thick{}{M}}$.

Finally recall that $M$ and $X$ are cocompact in $\Dm\A$ by Theorem~\ref{thm.cocomp}. Thus $X \in \thick{n}{M}$ follows from (the dual of) \cite[Proposition~2.2.4]{BvdB}.
\end{proof}

Now we are able to prove our result on strongly finitely generated categories between $\Db\E$ and $\Db\A$.

\begin{proof}[Proof of Theorem~\ref{thm.strongfg_general}]
Let $\mathcal{T}$ be a strongly finitely generated category, such that $\Db\E \subseteq \mathcal{T} \overset{\text{thick}}{\subseteq} \Db\A$. Let $T$ be a strong generator, that is there is $n \in \mathbb{Z}$ such that $\mathcal{T} = \thick{n}{T}$.

Since $\Db\E \subseteq \mathcal{T}$ in particular we have $E \in \thick{n}{T}$ for any $E \in \Db\E$. Hence, by Theorem~\ref{thm.anti-ghost}((1) $\implies$ (4)), we have that $(\G^M \!|_{\Db\E})^n(-, E) = 0$ for any such $E$, and therefore $(\G^M \!|_{\Db\E})^n = 0$.

Now let $X \in \Db\A$ arbitrary. Then, by Theorem~\ref{thm.anti-ghost}((4) $\implies$ (1)), we have $X \in \thick{n}{T} = \mathcal{T}$. It follows that $\mathcal{T} = \Db\A$.
\end{proof}

We conclude this paper with a consequence of Theorem~\ref{thm.generating_intro}(1), describing how the $\thick{i}{-}$ grow in the case of noetherian algebras.

\begin{cor}
Let $\Lambda$ be a noetherian algebra. For $M \in \Db\modL$ either the inclusions
\[ \thick{1}{\Lambda \oplus M} \subseteq \thick{2}{\Lambda \oplus M} \subseteq \thick{3}{\Lambda \oplus M} \subseteq \cdots \]
are all proper, or $\thick{n}{\Lambda \oplus M} = \Db\modL$ for some $n$.
\end{cor}

% =============================================================================

\end{document}